\documentclass[a4paper,11pt]{article}
\usepackage[utf8]{inputenc}
\usepackage[T1]{fontenc}

\usepackage{wrapfig}
\usepackage{amsthm,amsmath}
\usepackage{tikz}
\usetikzlibrary{graphs,graphs.standard}
\tikzgraphsset{edges={draw,semithick}, nodes={circle,draw,semithick}}
\usepackage{mathrsfs,amssymb,amsfonts} 
\usepackage{enumitem}
\usepackage{fullpage}
\usepackage{hyperref, enumerate}
\usepackage[babel]{microtype}
\usepackage[english]{babel}
\usepackage[capitalise]{cleveref}

\usepackage{thmtools}
\usepackage{mathtools, comment}
\usepackage{amssymb}
\usepackage[nomath]{lmodern}
\usepackage{graphicx}
\usepackage{pgf,tikz,tkz-graph,subcaption}
\usetikzlibrary{arrows,shapes}
\usetikzlibrary{decorations.pathreplacing}
\usepackage{tkz-berge}
\usepackage{enumitem}
\usepackage[normalem]{ulem}
\usepackage{hyperref}
\hypersetup{colorlinks = true, linkcolor = blue, citecolor = blue, urlcolor = blue}

\newcommand*{\ceilfrac}[2]{\mathopen{}\left\lceil\frac{#1}{#2}\right\rceil\mathclose{}}

\newcommand*{\abs}[1]{\lvert #1\rvert}

\newcommand{\diam}{diam}

\allowdisplaybreaks

\usepackage[margin=1in]{geometry}
\newtheorem{defi}{Definition}

\newtheorem{cor}[defi]{Corollary}
\newtheorem{thr}[defi]{Theorem}

\newtheorem{prop}[defi]{Proposition}
\newtheorem{exam}[defi]{Example}

\newtheorem{claim}[defi]{Claim}
\newcommand*{\myproofname}{Proof}
\newenvironment{claimproof}[1][\myproofname]{\begin{proof}[#1]}{\end{proof}}

\title{Abundancy of $z$-\v Solt\'es' digraphs}

\author{Stijn Cambie
 \thanks{Department of Computer Science, KU Leuven Campus Kulak-Kortrijk, 8500 Kortrijk, Belgium. Supported by a postdoctoral fellowship by the Research Foundation Flanders (FWO) with grant number 1225224N.} }

\begin{document}
\parindent=0cm
\maketitle

\begin{abstract}
 We prove the existence of infinitely many \v Solt\'es' digraphs, the digraph analogue of \v Solt\'es' graphs.
We also give an example of a \v Solt\'es' digraph with trivial automorphism group.
\end{abstract}

\section{Introduction}

As an analogue to \v Solt\'es graphs, which originate from \v Solt\'es' paper~\cite{Soltes91} from 1991, and the extended question~\cite[Prob.~3]{AOVVVY23}, we define a $z$-\v Solt\'es' digraph as a digraph $D$ for which the removal of any vertex decreases its total distance by exactly $z$. Here the total distance, $W(D)$, is the sum of distances (length of a shortest connecting path) over all order pairs of vertices.
When $z=0$, such a digraph is naturally named a \v Solt\'es' digraph.
When $z \le 0$, such a digraph is a negative-\v Solt\'es' digraph.

We try to strengthen the intuition behind the main question in~\cite[Ques.~13]{Cambie24}, which asks if there are infinitely many negative-\v Solt\'es' graphs with minimum degree at least $3$, by confirming it in the digraph case. We do so by first proving that there are infinitely many negative-\v Solt\'es' digraphs, and further strengthening this to the existence of infinitely many \v Solt\'es' digraphs.

\begin{thr}\label{thr:main}
    For every $z \in \mathbb Z,$ there are infinitely many digraphs $D$ for which $W(D)-W(D \setminus v)=z$ for every $v \in V.$
\end{thr}

\begin{cor}
    There are infinitely many \v Solt\'es' digraphs.
\end{cor}

The proof involves a specific class of vertex-transitive digraphs, circulant digraphs $D(n,S).$ Here the vertices are named with integers in $[n]=\{1,2,\ldots, n\}$, and there is a directed edge from $i$ to $j$ iff $j-i \in S \cup\{-1\}$ (considered modulo $n$).
Hereby we choose $n$ and $S$ carefully in a few steps.

\begin{exam}
Next to $D(11, \{1\}) \cong C_{11}$, also $D(85, \{4\})$ is \v Solt\'es' digraph (and even oriented graph). This one is depicted (on a torus) in~\cref{fig:D854}.
\end{exam}

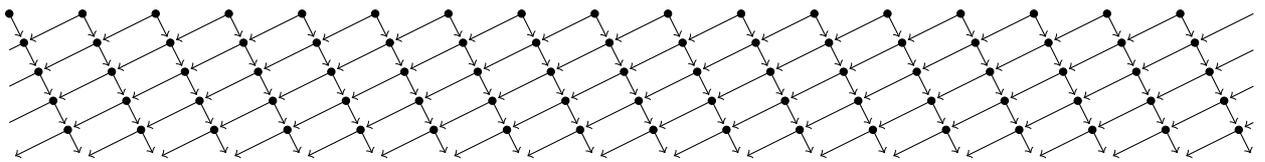
\begin{figure}[h]
    \centering

\begin{tikzpicture}[scale=0.385]
    \foreach \x in {0,1,2,3,4,5,6,7,8,9,10,11,12,13,14,15,16}{
        \foreach \y in {0,1,2,3,4}{
        
         \draw[fill] (2.5*\x+0.5*\y, 10-\y) circle (0.125);
        }
    }

    \foreach \x in {1,2,3,4,5,6,7,8,9,10,11,12,13,14,15,16}{
        \foreach \y in {0,1,2,3}{
        \draw[->] (2.5*\x+0.5*\y, 10-\y) -- (2.5*\x+0.5*\y-1.8, 9.1-\y);
        }
    }

    \foreach \x in {0,1,2,3,4,5,6,7,8,9,10,11,12,13,14,15,16}{
        \foreach \y in {0,1,2,3,4}{
        \draw[->] (2.5*\x+0.5*\y, 10-\y) -- (2.5*\x+0.5*\y+0.4, 9.2-\y);
        }
    }

    \foreach \x in {0,1,2,3,4,5,6,7,8,9,10,11,12,13,14,15,16}{
        \draw[->] (2+2.5*\x,6)--(0.2+2.5*\x,5.1);
    }

    \draw (0.5,9)--(0,8.75);
    \draw (1,8)--(0,7.5);
    \draw (1.5,7)--(0,6.25);

    \draw[->] (42.5,10)--(40.7,9.1);
    \draw[->] (42.5,8.75)--(41.2,8.1);
    \draw[->] (42.5,7.5)--(41.7,7.1);
    \draw[->] (42.5,6.25)--(42.2,6.1);
\end{tikzpicture}    
\caption{The \v Soltes' digraph $D(85,\{4\})$}\label{fig:D854}
\end{figure}

To have an idea that $W(D)=W(D \setminus v)$ for $D=D(85, \{4\})$ and a $v \in V(D)$, one can observe that after removing $v$ the distance from its left neighbour towards its right neighbour changes from $2$ to $22$ ($2 \equiv 21 \cdot -4+1 \pmod{85})$. There are fewer vertices and thus summands in the total distance of $D \setminus v$, but some of these distance are larger than they were within $D$.

Finally, we note that there are also non-vertex-transitive \v Solt\'es' digraphs.
Using the census from~\cite{POSV13}, by orienting its edges (and some edges in both directions), we can find e.g. (two) \v Solt\'es' digraphs which are $2$-regular ($\delta+=\delta^-=\Delta^+=\Delta^-=2$), bipartite and not vertex-transitive. See~\cite[\texttt{SoltesDigraphs/SoltesDigraph\_order960}]{C24}.
By combining multiple difference sets, we obtain an example of a \v Solt\'es' digraph of order $n=3306$, with a trivial automorphism group, which would disprove a digraph analogue of conjectures~\cite[Conj.~1.2]{BKS23} (and other beliefs) in the strongest possible sense.
This digraph is not regular in the sense that $\delta+=12<\Delta^+=14$ ($\delta^-=\Delta^-=13$).
See~\cite[\texttt{SoltesDigraphs/NonVTDigraph3306}]{C24}. By similar modifications, one may expect there to be infinitely many of these, implying the abundance.

In conclusion, these findings enforce intuition from~\cite{Cambie24} about the graph case, that the essence is the extremal question of the existence of infinitely many negative-\v Solt\'es' graphs.
If there is a clear abundance of negative-\v Solt\'es' graphs, we can expect an abundance of \v Solt\'es' graphs.

\section{Proofs on the abundancy of $z$-\v Solt\'es' digraphs}\label{sec:main}

We imagine the reader being familiar with the total distance of a graph $G=(V,E)$ and digraph $D=(V,A)$, and refer to the survey~\cite{KST23} on the Wiener index (total distance).

We use $\delta^+, \delta^-, \Delta^+, \Delta^-$ for the minimum and maximum out- and indegree of a digraph.
The in- and out-transmission of a vertex $v$ is denoted with $\sigma^-(v)= \sum_{u \in V} d(u,v)$ and $\sigma^+(v)= \sum_{u \in V} d(v,u).$
The total distance of a digraph equals $W(D)= \sum_{v \in V} \sigma^-(v)= \sum_{v \in V} \sigma^+(v)= \sum_{(u,v) \in V \times V} d(u,v)$

We first define the circulant digraphs in the following way (adding $-1$ to the difference set $S$ compared with the normal convention).
\begin{defi}
    For a subset $S \subseteq [n-2]=\{1,2,\ldots,n-2\}$, we  define the digraph $D(n, S)$ as the digraph with vertex set $V=[n]$ and directed edges $E=\{( i, i-1) \mid i \in [n] \} \cup \{( i, i+k) \mid i \in [n], k \in S \}.$
    Here numbers are interpreted modulo $n$.
\end{defi}

We can observe that this family leads to infinitely many vertex-transitive examples
for which $W(D)-W(D \setminus v)$ can be either positive, or negative. 

Here positive is trivial by considering dense digraphs, by the following digraph analogue of~\cite[Prop.~8]{KMS18} for which the short proof easily extends.

\begin{prop}
    Let $D$ be a digraph for which $\delta^-(D)+\delta^+(D)\ge n\ge 4.$
    Then both $D$ and $D \setminus v$ have diameter $2$ and consequently $W(D)-W(D \setminus v)>0$.
\end{prop}

\begin{proof}
    Let $u,u'$ be any two arbitrary vertices for which $\vec { uu'} \not in A$ (otherwise $d(u,u')=1$) and take any vertex $v \not \in \{u,u'\}$. By pigeon hole principle there are at least two vertices $w,w'$ such that $u$ is an in-neighbour from them and $u'$ is an out-neighbour, since $\abs{ N^-(u) \cap N^+(u')} \ge \delta^-(u)+\delta^+(u)-(n-2)\ge 2.$
    Hence after removing a vertex, there is at least one of $w,w'$ present and $d(u,u') =2.$
\end{proof}

Related to this, in the appendix,~\cref{prop:diam_ge3}, we prove the stronger statement that diameter two implies that at least one vertex $v$ satisfies $W(D)-W(D \setminus v)>0$.

The infinitude of examples where the difference is negative, is proven in the following proposition.
Note that here we do not care about the number of such non-isomorphic digraphs. 
That would be hard due to counterexamples to the digraph analogue of Adam's conjecture~\cite{AP79}. But one can imagine that there are roughly $2^{\Theta(n^{0.5})}$ of them, which is a huge number.

\begin{prop}
    If $\max S \le 1/9 n^{1/2}$, for $n$ sufficiently large, $W(D_{n,S})-W(D_{n,S} \setminus v)<0$ for $v \in V(D_{n,S} ).$
\end{prop}

\begin{proof}
    Let $m=\max S$.
    Let $v$ be the vertex which corresponds to the number $n$.
    Note that $\sigma^+(v)=\sigma^-(v) \le n \cdot \diam D \le n \left( \frac{n}{m}+m \right)< 9.5n^{3/2}$.

    On the other hand, the shortest directed path from $0<i<2 \sqrt n$ to $n-j$ (with $0<j<2 \sqrt n$) had length $i+j$ within $D$, but is at least of length $\frac{n-i-j}{m}$ in $D \setminus v$.
    Hence the distances between those pairs increase by at least $(5-o(1))\sqrt n.$
    This is a total increase of at least $
    (20-o(1))n^{3/2}$, which is larger than $\sigma^+(v)+\sigma^-(v).$
    Equivalently, $W(D_{n,S})<W(D_{n,S} \setminus v).$    
\end{proof}

Knowing that $W(D)-W(D \setminus v)$ can be both positive and negative, by an abundance of examples, we now prove the main result,~\cref{thr:main}, that every exact difference can appear infinitely often.
It may also be clear that one cannot aim for an exact characterization as wondered about in~\cite[Prob.~3]{AOVVVY23}.

\begin{proof}[Proof of~\cref{thr:main}]
    We will prove that for every $z \in \mathbb Z,$ and every sufficiently large (as a function of $z$) positive integer $m$, there is a digraph $D$ of the form $D(n,S)$ with $\max S=m$ and $W(D)-W(D \setminus v)=z$ for every $v \in V(D).$  So fix $m$ large (in particular, larger than $\abs{z})$.

    Hereby we first choose $n$ such that the difference is near $z$ when considering $D(n,[m])$.

    \begin{claim}\label{clm:roughD1}
        There is a choice $n \sim 6m^2$ such that $D(n,[m])$ satisfies 
        $z- 9m \le W(D)-W(D-v) \le z-3$.
    \end{claim}
    \begin{claimproof}
        Write $n=a(m+1)+1+r$, where $0 \le r \le m.$
        Let $D=D(n,S),$ where $S=[m].$
Then \begin{equation}\label{eq:transmissions}
\sigma^+(v)=\sigma^-(v)=(m+1) \binom{a+1}2+r(a+1).\end{equation}
The sum of the extra distance between pairs of vertices in $V \setminus v$, $\sum_{\{u,w\} \subset V \setminus v} d_{D \setminus v}(u,w)-d_{D}(u,w)$, equals 
\begin{equation}\label{eq:sum_detours}
    W(D-v)-\frac{n-2}{n}W(D)= \sum_{i=1}^a \sum_{j=1}^{a-i} \left( \ceilfrac{n-j-i}{m}-(j+i)\right)+\binom{a}{2} = \frac{m+1}{m}\binom{a}{3}+O(a^2).
\end{equation}
The latter is the sum of detours between elements near $n$, as well as the little detour between elements $n-j$ and $i$ where both $i$ and $j$ are multiples of $m$ and so the shortest distances becomes larger after removing $v=n.$

If the shortest path in $D$ from $i$ to $n-j$ (with $n-j>i$) uses $n$, it is because all arcs go one down (correspond with difference $-1$).
Thence $d_D(i,n-j)=i+j$ and $d_{D \setminus v}(i,n-j)=\ceilfrac{n-j-i}{m}.$ The latter is strictly larger than the former exactly when $i+j \le a.$

If every shortest path from $n-j$ to $i$ (with $n-j>i$) uses $n$, it is because it is unique and all its directed edges go up with $m$.
Thus $d_D(i,n-j)=\frac{i+j}{m}$, $d_{D \setminus v}(i,n-j)=d_D(i,n-j)+1,$ and 
$n-(i+j)>\frac{i+j}{m}$.
This implies that $i/m+j/m\le a$. There are $\sum_{i'=1}^{a-1} (a-i')=\binom a2$ many pairs of positive integers $(i/m, j/m)$ satisfying this.

Comparing~\eqref{eq:transmissions} and~\eqref{eq:sum_detours} and $\sigma^+(v), \sigma^-(v)$, we conclude that $a\sim 6m$ to ensure that $W(D-v) \sim W(D)$.
More precisely (see~\cref{sec:app2} for the computation), if $a=6m-1$ and $r=m,$ 
$W(D-v)- W(D)= 36m^3+\frac{91}{2}m^2+O(m) - 36m^3-42m^2-O(m)=\frac{7}2m^2-O(m)>z.$

If $a=6m-1$ and $r=0,$ then the analoguous computation gives 

$W(D-v)- W(D)= 36m^3+\frac{55}{2}m^2+O(m) - 36m^3-30m^2-O(m)=-\frac{5}2m^2-O(m)<z-9m.$

Changing $r$ from $n$ to $0$ by one at a time, the value of~\eqref{eq:sum_detours} changes by at least $(a-m)+\ldots+(a-5m) \sim 15m$ and at most $a+(a-m)+\ldots+(a-5m) = 21m-6$, while $\sigma^+(v)+\sigma^-(v)$ changes by $2(a+1)=12m.$
Thus there is a choice of $r$ (and $n=(6m-1)(m+1)+1+r$) for which 
$W(D-v) -W(D)$ satisfies the condition of the claim.    \end{claimproof}

Next, we discard some large elements from $[m]$

    \begin{claim}\label{clm:roughD2}
        There is value $\ell=o(m)$ such that $D(n,[m-\ell] \cup \{ m-1,m\})$ satisfies 
        $z- o(m) \le W(D)-W(D-v) \le z$ and $W(D)-W(D-v)-z$ is even.
    \end{claim}
    \begin{claimproof}
        We choose $n,m$ such that $D(n,[m])$ satisfies~\cref{clm:roughD1}.
        Note that the formula~\eqref{eq:sum_detours} is still valid, except that there are $\ell-2$ choices for $j$, which are exactly $(m-1)j'$ where $j' \in [\ell-2]$, such that there is a unique shortest path from $n-j$ to $m-1$ (with directed edges corresponding with difference $m-1$). For these, the length of a shortest path goes up by one.

        On the other hand, the distance from $v$ to $i$ increases by one when 
        $km-\ell<i<k(m-1)$ for some $k \ge 1.$ This implies that the formula~\eqref{eq:transmissions} goes up by $\binom{\ell} 2$.

        Thus we can choose $\ell = O( \sqrt m)$ such that the estimate of this lemma is satisfied.
        Note hereby that $2 \binom{\ell}{2}-\ell=\ell(\ell-2)$ has the parity of $\ell$ and it is $0,3$ (both bounded by $3$) for $\ell \in \{2,3\}$.
    \end{claimproof}

    Finally, we discard $\frac{W(D-v)+z-W(D)}{2}$ many odd elements from $S=[m-\ell] \cup \{ m-1,m\}$, where each of these odd elements is bounded by $\frac{m- \ell-2}{2}$.
    Doing this, there are no additional unique shortest paths created that use $n$ (all lengths $\ell'$ of the directed edges need to be the same, but either $2 \ell'$, both $\ell'+1, \ell'-1$ or both $\ell' \pm 2$ belong to $S$).
    Both the in- and out-transmission go up by one deleting one such odd element. Since $\sigma^+(v)+\sigma^-(v)$ goes up by two by one removal, at the end we obtained a set $S$ for which $W(D)-W(D-v)= z.$
\end{proof}



\bibliographystyle{abbrv}
\bibliography{ref}

\begin{thebibliography}{10}

\bibitem{AOVVVY23}
M.~Akhmejanova, K.~Olmezov, A.~Volostnov, I.~Vorobyev, K.~Vorob'ev, and Y.~Yarovikov.
\newblock Wiener index and graphs, almost half of whose vertices satisfy \v solt\'es property.
\newblock {\em Discrete Appl. Math.}, 325:37--42, 2023.

\bibitem{AP79}
B.~Alspach and T.~D. Parsons.
\newblock Isomorphism of circulant graphs and digraphs.
\newblock {\em Discrete Math.}, 25:97--108, 1979.

\bibitem{BKS23}
N.~{Ba{\v{s}}i{\'c}}, M.~{Knor}, and R.~{{\v{S}}krekovski}.
\newblock {On regular graphs with {\v{S}}olt{\'e}s vertices}.
\newblock {\em Ars Mathematica Contemporanea}.

\bibitem{C24}
S.~Cambie.
\newblock Code and data related to \v solt\'es' problem.
\newblock https://github.com/StijnCambie/Soltes, 2024.

\bibitem{Cambie24}
S.~Cambie.
\newblock {Towards the essence of \v Solt\'es' problem}.
\newblock {\em arXiv e-prints}, page arXiv:2406.03451, June 2024.

\bibitem{Cambie24Hyper}
S.~{Cambie}.
\newblock {{\v{S}}olt{\'e}s' hypergraphs}.
\newblock {\em arXiv e-prints}, page arXiv:2406.01504, June 2024.

\bibitem{KMS18}
M.~Knor, S.~Majstorovi{\'c}, and R.~{\v{S}}krekovski.
\newblock Graphs preserving {Wiener} index upon vertex removal.
\newblock {\em Appl. Math. Comput.}, 338:25--32, 2018.

\bibitem{KST23}
M.~Knor, R.~Škrekovski, and A.~Tepeh.
\newblock Selected topics on wiener index, 2023.

\bibitem{POSV13}
P.~Poto{\v{c}}nik, P.~Spiga, and G.~Verret.
\newblock Cubic vertex-transitive graphs on up to 1280 vertices.
\newblock {\em J. Symb. Comput.}, 50:465--477, 2013.

\bibitem{Soltes91}
L.~{\v{S}}olt{\'e}s.
\newblock Transmission in graphs: a bound and vertex removing.
\newblock {\em Mathematica Slovaca}, 41(1):11--16, 1991.

\end{thebibliography}

\section*{Appendix}\label{sec: appendix}

\appendix

\section{Few negative \v Soltes' digraphs}\label{sec:app1}

In this part, we consider digraphs with diameter $2$ and prove the analogue of~\cite[Prop.~4]{Cambie24}. Note that almost all digraphs of order $n$ have diameter $2$, considering the Erd\H{o}s-Renyi random graph model for digraphs. The proof combines the ideas in the proofs from~\cite[Prop.~4]{Cambie24} and~\cite[Prop.~6]{Cambie24Hyper}.

\begin{prop}\label{prop:diam_ge3}
 A digraph $D$ of order $n>1$ with $\diam(D)\le 2$ has at least one vertex $v$ for which either $D \setminus v$ is disconnected, or $W(D \setminus v)<W(D).$
\end{prop}

\begin{proof}
    Assume $D$ would be a counterexample to the statement (which trivially needs to satisfy $n>4$).
    First notice that $\delta^-(D)$ and $\delta^+(D)$ are both at least two, since otherwise removing the vertex from the unique incoming edge or unique outgoing edge, disconnects the digraph.
    By the handshaking lemma, the number of arcs is thus at least $2n.$
    For a pair of vertices $u,u',$ there can be at most one vertex $w$ in $V \setminus \{ u,u'\}$ whose deletion increases $d(u,u').$
    In that case, $\vec{uu'}$ is not an arc of $D,$ while $\vec{uw}$ and $\vec{wu'}$ are.
    In this case, associate $w$ to the ordered pair $(u,u').$
    
    Since there are $n(n-1)-\abs{A(D)} \le n(n-3)$ ordered pairs which are not a directed edge, there is a vertex $v$ associated with at most $n-3$ pairs.

    For each pair $(u,u')$ to which it is associated, we know that there is at least one of the $n-3$ neighbours $w$ in $V \setminus \{ v,u,u'\}$ for which both ordered pairs $(u,w)$ and $(w,u')$ were not associated to $v$.
    Hence $d_{D \setminus v}(u,u') \le d_{D \setminus v}(u,w)+d_{D \setminus v}(w,u') \le 4.$
    Thus there are at most $n-3$ pairs for which the distance goes up by at most $2.$
    The latter is smaller than $\sigma^+(v)+\sigma^-(v)\ge 2(n-1)$, i.e., $W(D \setminus v)<W(D).$
\end{proof}

\section{Detailed computations}\label{sec:app2}

If $a=6m-1$ and $r=m,$ thus $n=(a+1)(m+1)$, and using $s=i+j-1$, we obtain

\begin{align*}
     \sum_{i=1}^a \sum_{j=1}^{a-i} \left( \ceilfrac{n-j-i}{m}-(j+i)\right)&= 
    \sum_{s=1}^{a-1} s \ceilfrac{(m+1)(a-s)}{m}\\   
    &= 
    \sum_{s=1}^{a-1} s(a-s) + \sum_{s=1}^{a-1} s \ceilfrac{(a-s)}{m}\\
    &=\binom{a+1}{3}+\sum_{k=0}^4 \sum_{s=km-1}^{(k+1)m-2} (6-k)s + \sum_{s=1}^{m-2} 6s\\
    &=\binom{6m}{3}+\sum_{k=1}^6 \binom{km-1}{2}\\
    &= 36m^3+\frac{55}{2}m^2+O(m).
\end{align*}

If $a=6m-1$ and $r=0,$ thus $n=a(m+1)+1$, and using $s=i+j$, we obtain exactly the same sum up to $\binom a2$,

\begin{align*}
     \sum_{i=1}^a \sum_{j=1}^{a-i} \left( \ceilfrac{n-j-i}{m}-(j+i)\right)&= 
    \sum_{s=2}^{a} (s-1) \ceilfrac{(m+1)(a-s)+1}{m}\\   
    &= 
    \sum_{s=2}^{a} (s-1)(a-s) + \sum_{s=2}^{a} (s-1) \ceilfrac{(a-s+1)}{m}\\
    &=\binom{a}{3}+\sum_{s'=1}^{a-1} s' \ceilfrac{(a-s')}{m}\\
    &=-\binom a2 + \binom{a+1}{3}+\frac{91}2m^2+O(m)\\
    &= 36m^3+\frac{19}{2}m^2+O(m).
\end{align*}

\section{A \v Soltes' digraph with trivial automorphism group.}\label{sec:app3}

For $m=23$, executing the procedure from~\cref{sec:main}, gives us e.g. $a=6m-1, r=17$ and
and $n=a(m+1)+1+r=3306.$ 
After this, we can find multiple sets $S$ for which the vertex-transitive digraph $D(n,S)$ satisfies \v Soltes' property.

This is the case for
$S_1=\{ 2, 4,  6, 8, 10, 9, 12, 13, 14, 15, 22, 23\},
S_2=\{ 2, 4,  6, 8, 10, 7, 12, 13, 14, 15, 22, 23\}$ and $
S_3=\{ 2, 4,  6, 8, 10, 11, 12, 13, 14, 15, 22, 23\}$.

By choosing different difference sets to construct the out-neighbours of different vertices in a good way (not every combination does the job), we can end up with a \v Soltes' digraph with $n$ vertex orbits.

We did this, for the digraph on $V=[n]$ and directed edges

$\{ (i,i+k) \mid k \in S_1, i \equiv 0 \pmod 3 \vee i \in \{ 7,  13,  19\} \} \cup \{ (i,i+k) \mid k \in S_3, i \equiv 2 \pmod 3 \}$
$\cup \{ (i,i+k) \mid k \in S_2, i \equiv 1 \pmod 3 \wedge i \not \in \{ 7,  13,  19\} \}.  $

\end{document}